\documentclass{amsart} 
\usepackage{amsmath,amsthm,amssymb,amsfonts,amscd}
\usepackage{xcolor}
\usepackage[dvips]{graphicx}
\usepackage[all]{xy}
\usepackage{xcolor}

\allowdisplaybreaks

\theoremstyle{plain}
\newtheorem{theorem}{Theorem}[section]
\newtheorem{proposition}[theorem]{Proposition}
\newtheorem{lemma}[theorem]{Lemma}
\newtheorem{corollary}[theorem]{Corollary}

\theoremstyle{definition}
\newtheorem*{definition}{Definition}

\theoremstyle{remark}
\newtheorem{remark}[theorem]{Remark}

\begin{document}

\title[A Cameron--Storvick type theorem on ${C_{a,b}^2[0,T]}$  with   applications]
{A Cameron--Storvick type theorem on $C_{a,b}^2[0,T]$  with   applications}

\author[J.G. Choi]{Jae Gil Choi}
\address{Jae Gil Choi \\ School of General Education \\ Dankook University \\ Cheonan 31116, Republic of Korea}
\email{jgchoi@dankook.ac.kr}

\author[D. Skoug]{David Skoug}
\address{David Skoug \\ Department of Mathematics \\ University of Nebraska-Lincoln  \\ Lincoln, NE 68588-0130, USA}
\email{dskoug@math.unl.edu}



\begin{abstract}
The purpose of this paper is to establish a  very general  Cameron--Storvick 
theorem involving the  generalized analytic Feynman integral of functionals on the product 
function space $C_{a,b}^2[0,T]$. The  function space $C_{a,b}[0,T]$ can be induced 
by the generalized Brownian motion process associated with continuous  functions $a$ 
and $b$. To do this we first introduce the class $\mathcal F_{A_1,A_2}^{\,\,a,b}$ of 
functionals on $C_{a,b}^2[0,T]$ which is a generalization of the Kallianpur and 
Bromley  Fresnel class $\mathcal F_{A_1,A_2}$. We  then proceed to establish a 
Cameron--Storvick type theorem on the product function space $C_{a,b}^2[0,T]$. 
 Finally we use our Cameron--Storvick type theorem to obtain several 
meaningful results and examples.  
\end{abstract}

\maketitle

\setcounter{equation}{0}
\section{Introduction}\label{sec:1}

\par
Let $C_0[0,T]$ denote one-parameter Wiener space, that is, the space
of  all  real-valued continuous functions $x$ on $[0,T]$ with $x(0)=0$.
Let $\mathcal{M}$ denote the class of  all Wiener measurable subsets
of $C_0[0,T]$ and let $m_w$ denote Wiener measure.  
Then, as is  well-known,  $(C_0[0,T],\mathcal{M},m_w)$ is 
a complete  measure space. 
 
\par
In \cite{Cameron51} Cameron  established an integration by parts formula 
for the Wiener measure $m_w$. 
More precisely,  in \cite{Cameron51}, Cameron introduced the first 
variation (a kind of  G\^ateaux  derivative) of functionals  on the
classical Wiener space $C_0[0,T]$ and established a formula  involving  
the Wiener integral of the first variation.
 This was the first  infinite dimensional 
integration by parts formula.   In \cite{Donsker} Donsker also established 
this formula using a different method, and applied it to study Fr\'echet--Volterra 
differential equations. In \cite{Kuo74,KL11}   Kuo and Lee established 
an integration by parts formula for abstract Wiener space which they 
  then used to evaluate various functional integrals.

\par
In \cite{CS91}, Cameron 
and Storvick  also established an integration by parts formula 
involving the  analytic  Feynman integral of  functionals on $C_0[0,T]$. 
They also applied their  celebrated parts formula to establish the existence 
of the analytic  Feynman integral of unbounded  functionals on $C_0[0,T]$.
The integration by  parts formula on $C_0[0,T]$  suggested 
in \cite{CS91} was improved in \cite{PSS98} to study the  parts formulas 
involving the analytic Feynman integral and the analytic Fourier--Feynman transform. 
Since then the parts formula for the analytic Feynman integral is called the
Cameron--Storvick theorem by many mathematicians.

\par
The Cameron--Storvick type theorem and  related topics were also  developed  for 
functionals on the very general function  space  $C_{a,b}[0,T]$  in  \cite{CCS03,CS03}. 
The function space $C_{a,b}[0,T]$, induced by  
generalized Brownian motion process (GBMP), was introduced by J. Yeh \cite{Yeh71,Yeh73}.

\par
A GBMP on a probability space $(\Omega,\Sigma, P)$ and a time interval $[0,T]$ 
is a Gaussian  process  $Y \equiv\{Y_t\}_{t\in [0,T]}$ such that $Y_0=c$ almost
surely for some constant $c \in \mathbb  R$, and for any set of time  
moments $0= t_0 < t_1< \cdots<t_n \le T$ and any Borel set $B\subset \mathbb R^n$, 
the measure $P(I_{t_1,\ldots,t_n,B})$ of the cylinder set $I_{t_1,\ldots,t_n,B}$ of 
the form 
\[
I_{t_1,\ldots,t_n,B} 
=\big\{\omega\in \Omega:(Y(t_1,\omega),\ldots,Y(t_n,\omega))\in B \big\}
\] 
is equal to
\[
\begin{aligned}
&\bigg( (2\pi )^n \prod\limits_{j=1}^n\big(b(t_j)-b(t_{j-1})\big) \bigg)^{-1/2} \\ 
&\quad \times\int_{B}
\exp \bigg\{- \frac 12 \sum\limits_{j=1}^n
   \frac {((\eta_j-a(t_j))-(\eta_{j-1}-a(t_{j-1})))^2}
         {b(t_j)-b(t_{j-1})} \bigg\}
d\eta_1\cdots d\eta_n, 
\end{aligned}
\]
where $\eta_0=c$,    $a(t)$ is a continuous real-valued function on  $[0,T]$, 
and $b(t)$ is a  increasing continuous real-valued function on $[0, T]$. 
Thus, the  GBMP $Y$ is determined by the continuous functions $a(\cdot)$ 
and $b(\cdot)$. For more details, see \cite{Yeh71,Yeh73}. Note that when $c=0$, 
$a(t)\equiv 0$ and $b(t)=t$ on $[0,T]$, the GBMP is  the standard Brownian motion  
(Wiener process).
In this paper we set $c = a(0)=b(0)=0$. Then the function space $C_{a,b}[0,T]$ 
induced by the GBMP $Y$ determined by the  $a(\cdot)$ and $b(\cdot)$  
can be considered as the space of continuous sample paths of $Y$. 

\par
In  this paper, we  establish a  general  Cameron--Storvick 
theorem involving the  generalized analytic Feynman integral of functionals 
on the product  function space $C_{a,b}^2[0,T]$.  We also present a meaningful 
example  to which our  Cameron--Storvick type theorem can be applied.  
To do this, we  introduce the class $\mathcal F_{A_1,A_2}^{\,\,a,b}$ 
of functionals on $C_{a,b}^2[0,T]$ which is a generalization of the Kallianpur and  Bromley Fresnel 
class $\mathcal F_{A_1,A_2}$,  see \cite{KB84}.

\setcounter{equation}{0}
\section{Preliminaries}\label{sec:2}

In this section  we first give a brief background of some ideas and results
which are needed to establish our new results in Sections  \ref{sec:CS-thm},
\ref{sec:fab}, and \ref{sec:last} below.

\par
Let  $a(t)$ be an absolutely continuous real-valued  function on $[0,T]$
with $a(0)=0$ and  $a'(t)\in L^2[0,T]$,  and let $b(t)$ be a strictly
increasing, continuously differentiable real-valued function with $b(0)=0$
and $b'(t)>0$ for each $t \in [0,T]$. The generalized Brownian motion
process $Y$ determined by $a(t)$ and $b(t)$ is a Gaussian process  with
mean function  $a(t)$ and covariance function $r(s,t)=\min\{b(s),b(t)\}$.

\par
We consider the  function space $(C_{a,b}[0,T],\mathcal W(C_{a,b}[0,T]),\mu)$
induced by the generalized Brownian motion process $Y$, where $C_{a,b}[0,T]$
denotes the set of continuous sample paths of the generalized Brownian motion
process $Y$ and  $\mathcal W(C_{a,b}[0,T])$ is the $\sigma$-field of all
$\mu$-Carath\'eodory measurable subsets  of $C_{a,b}[0,T]$.
For the precise procedure used to construct this function space, we refer to the
references \cite{CC16,CCK15,CCS03,CCS19,CS03,Yeh71,Yeh73}.

\par
We note that  the coordinate process
defined by  $e_t(x)=x(t)$ on $C_{a,b}[0,T]\times [0,T]$ is also the generalized
Brownian motion process determined by $a(t)$ and $b(t)$, i.e., for each $t\in[0,T]$,
$e_t(x)\sim N(a(t),b(t))$, and the process $\{e_t: 0\le t\le T\}$  has nonstationary
and independent increments.

\par
Recall that the process $\{e_t: 0\le t\le T\}$ on $C_{a,b}[0,T]$ is a continuous process.
Thus the function space $C_{a,b}[0,T]$ reduces to the classical  Wiener space $C_0[0,T]$,
considered in papers \cite{Cameron51,CS91,Johnson82,JS79,PSS98}
if and only if $a(t)\equiv 0$ and $b(t)=t$ for all $t\in[0,T]$.

\par
In \cite{CCS03,CCS10,CS03}, the generalized analytic Feynman integral
and the generalized analytic Fourier--Feynman transform of
functionals on $C_{a,b}[0,T]$ were investigated. The functionals considered
in \cite{CCS03,CCS10,CS03} are associated with the separable Hilbert space
\[
L_{a,b}^2[0,T]
=\bigg\{   v :  \int_{0}^{T} v^2 (s) d b(s)  <+\infty  \hbox{ and }
\int_0^T v^2 (s) d|a|(s)   < +\infty \bigg\}
\]
where $|a|(\cdot)$ denotes the total variation   function of $a(\cdot)$.
The inner product on $L_{a,b}^2[0,T]$ is given by the formula
\[
(u,v)_{a,b}
=\int_0^T u(s)v(s)dm_{|a|,b}(s)
\equiv \int_0^T u(s)v(s)d[b(s)+|a|(s)]
\]
where $m_{|a|,b}$ is the Lebesgue--Stieltjes measure induced by the increasing
function $|a|(\cdot)+b(\cdot)$ on $[0,T]$. We 
note that $\| u\|_{a,b}\equiv\sqrt{(u,u)_{a,b}} =0$ 
if and only if $u(t)=0$ a.e. on $[0,T]$.

\par
The following linear   subspace of $C_{a,b}[0,T]$
plays an important role throughout this paper.

\par
Let
\[
C_{a,b}'[0,T]
 =\bigg\{ w \in C_{a,b}[0,T] : w(t)=\int_0^t z(s) d b(s) 
       \hbox{  for some   } z \in L_{a,b}^2[0,T]  \bigg\}.        
\]
For $w\in C_{a,b}'[0,T]$, 
with $w(t)=\int_0^t z(s) d b(s)$ for $t\in [0,T]$, 
let $D : C_{a,b}'[0,T] \to L_{a,b}^2[0,T]$ 
be defined by the formula 
\begin{equation}\label{eq:Dt}
D w(t)= z(t)=\frac{w'(t)}{b'(t)}. 
\end{equation}
Then $C_{a,b}' \equiv C_{a,b}'[0,T]$ with inner product
\[
(w_1, w_2)_{C_{a,b}'}
= \int_0^T  D w_1(s)  D w_2(s)d b(s)  
=\int_0^T z_1(s)z_2(s) d b(s)    
\] 
is a separable  Hilbert space.

\par
Note that  the two separable Hilbert spaces $L_{a,b}^2[0,T]$ and $C_{a,b}'[0,T]$
are (topologically)  homeomorphic under the linear operator given by equation \eqref{eq:Dt}.
The inverse operator of $D$ is given by
\begin{equation}\label{eq:D-inverse}
(D^{-1}z)(t)=\int_0^t z(s) d b(s),\quad t\in[0,T].
\end{equation}
 For a more detailed study of the inverse operator $D^{-1}$
of $D$, see \cite{CCS10-pan}.

 Recall that above, 
as well as in papers \cite{CCS03,CCS10,CS03},
we require that $a:[0,T]\to \mathbb R$ 
is an absolutely continuous function 
with $a(0)=0$ and with $\int_0^T |a'(t)|^2 d t<\infty$.
Our conditions on $b:[0,T]\to \mathbb R$
imply that $0<\delta <b'(t)<   M$
for some positive real numbers $\delta$ and $M$, and  all $t\in[0,T]$.
In this paper, in addition to the conditions put on $a(t)$ above, we now add the condition
\begin{equation}\label{eq:condi-mean01} 
\int_0^T |a'(t)|^2 d|a|(t)< +\infty.
\end{equation}
One can see that
the function $a: [0,T]\to\mathbb R$ satisfies 
the condition  \eqref{eq:condi-mean01} if and only if $a(\cdot)$ is an element of $C_{a,b}'[0,T]$.
Under the  condition \eqref{eq:condi-mean01}, we also observe that for each  $w\in C_{a,b}'[0,T]$ with $Dw=z$,
\[
(w,a)_{C_{a,b}'}=\int_0^T Dw(t) Da(t) db(t)=\int_0^T z(t)a'(t)dt=\int_0^T z(t)da(t).
\]

\par
 For each $w\in C_{a,b}'[0,T]$, the  Paley--Wiener--Zygmund (PWZ)  stochastic integral 
$(w,x)^{\sim}$ is given by the formula
\begin{equation}\label{eq:pwz-ab}
(w,x)^{\sim}
=\lim_{n\to\infty}\int_{0}^{T}\sum_{j=1}^{n}(w,g_{j})_{C_{a,b}'}Dg_{j}(t)dx(t)
\end{equation}
for $\mu$-a.e. $x\in C_{a,b}[0,T]$ where $\{g_j\}_{j=1}^{\infty}$ is a complete 
orthonormal set of functions in $C_{a,b}'[0,T]$ such that for each $j\in\mathbb N$, 
$Dg_j$ is of bounded variation on $[0,T]$. 
 For a more detailed study of the  space $C_{a,b}'[0,T]$ and the PWZ stochastic integral
given by \eqref{eq:pwz-ab}, see \cite{CCK15,CCS19}.

\par
In \cite{PS08}, Pierce and Skoug used the  
the inner product $(\cdot,\cdot)_{a,b}$ on $L_{a,b}^2[0,T]$ rather than  
the inner product $(\cdot,\cdot)_{C_{a,b}'}$ 
on $C_{a,b}'[0,T]$  to study the PWZ stochastic integral
and the related  integration formula  on  the function space $C_{a,b}[0,T]$.

\par
The  generalized analytic  Feynman integral on the function space $C_{a,b}[0,T]$
was  first defined and studied in \cite{CCS03,CS03}, and the  study of this 
integral has continued in \cite{CC16,CCK15,CCS10,CC12}. 
We assume  familiarity with
  \cite{CCK15,CCS03,CCS10,CS03,CC12}  and adopt the concepts and the definitions 
of the  generalized analytic  Feynman integral on   $C_{a,b}[0,T]$.

Based on the references \cite{JL00,JS79,KB84}, we present several concepts
which involve the scale-invariant measurability  to define 
a generalized analytic Feynman integral of functionals on the product  space 
$C_{a,b}^2[0,T]$. 

Let $(C_{a,b}^2[0,T],\mathcal W (C_{a,b}^2[0,T]),\mu^2)$
be the product function space,  where\\ $C_{a,b}^2[0,T]=C_{a,b}[0,T]\times C_{a,b}[0,T]$,
$\mathcal W (C_{a,b}^2[0,T])
\equiv \mathcal W (C_{a,b}[0,T]) \otimes \mathcal W (C_{a,b}[0,T])$
denotes the $\sigma$-field generated by  measurable rectangles $A\times B$
with $A,B\in \mathcal W (C_{a,b}[0,T])$, and  $\mu^2\equiv \mu\times\mu$.
A subset $B$ of   $C_{a,b}^{2}[0,T]$ 
is said to be  scale-invariant measurable
provided  $\{(\rho_1 x_1, \rho_2 x_2): (x_1, x_2)\in B\}$ 
is $\mathcal{W}(C_{a,b}^2[0,T])$-measurable
for every $\rho_1>0$ and $\rho_2>0$, 
and a scale-invariant measurable subset  $N$ of $C_{a,b}^{2}[0,T]$ 
is said  to be scale-invariant null  provided 
$(\mu\times\mu)(\{(\rho_1 x_1, \rho_2 x_2): (x_1, x_2)\in N\} )=0$
for every $\rho_1>0$ and $\rho_2>0$. 
A property that holds except on a scale-invariant null set 
is  said to hold s-a.e. on $C_{a,b}^2[0,T]$.
A functional $F$ on  $C_{a,b}^2[0,T]$ is said to be 
scale-invariant measurable  provided 
$F$ is defined on a scale-invariant measurable set 
and $F(\rho_1\,\,\cdot\,, \rho_2\,\,\cdot\,)$ 
is $\mathcal{W}(C_{a,b}^2[0,T])$-measurable 
for every $\rho_1>0$ and $\rho_2>0$.
If two functionals $F$ and $G$ defined on $C_{a,b}^2[0,T]$ 
are equal s-a.e., 
then we write $F\approx G$.

\par
We denote the product function space integral of a 
$\mathcal{W}(C_{a,b}^2[0,T])$-measurable functional $F$ by
\[ 
E[F]
  \equiv E_{\vec x}[F(x_1, x_2)]
  = \int_{C_{a,b}^2[0,T]}F(x_1,x_2)d(\mu\times\mu) (x_1,x_2)
\]
whenever the integral exists.

\par
Throughout this paper,
let $\mathbb C$, ${\mathbb C}_+$ and $\widetilde{\mathbb C}_+$ 
denote the complex numbers, the complex numbers with positive real part
and the nonzero complex numbers  with nonnegative real part, 
respectively.
Furthermore, for each $\lambda \in \widetilde{\mathbb C}_+$, 
$\lambda^{1/2}$ denotes the principal square root of $\lambda$; i.e., $\lambda^{1/2}$ is 
always chosen to have positive  real part, so that $\lambda^{-1/2}=(\lambda^{-1})^{1/2}$ is in $\mathbb C_+$. 
We also  assume that every functional $F$ on $C_{a,b}^2[0,T]$ 
we consider is s-a.e. defined and is scale-invariant measurable.

The following definition is due to Choi, Skoug and Chang \cite{CCS19,CSC13}
 
\begin{definition}\label{def:gfi}
Let $\mathbb{C}_+^2=\{ \vec\lambda =(\lambda_1,\lambda_2) \in \mathbb C^{2} : 
                  \mathrm{Re} (\lambda_j)  > 0 \hbox{ for } j=1,2 \}$ 
and let
$\widetilde{\mathbb C}_+^2=\{ \vec\lambda=(\lambda_1 , \lambda_2) \in \mathbb C^{2} : 
                  \lambda_j \ne 0 \hbox{ and }  \mathrm{Re} (\lambda_j)\ge 0    
                  \hbox{ for } j=1,2 \}$.
Let $F: C_{a,b}^2[0,T] \to \mathbb C $ be a scale-invariant measurable functional such that 
the function space integral
\[
J(\lambda_1, \lambda_2)=
 \int_{C_{a,b}^2[0,T]}F(\lambda_1^{-1/2}x_1, \lambda_2^{-1/2}x_2) 
   d (\mu\times\mu)(x_1, x_2)\\
\]
exists and is finite for each $\lambda_1>0$ and $\lambda_2>0$.
If there exists a function $J^*(\lambda_1,\lambda_2)$  
analytic on ${\mathbb C}_+^2$ such that  
$J^*(\lambda_1,\lambda_2)=J(\lambda_1,\lambda_2)$ 
for all $\lambda_1>0$ and $\lambda_2 >0$,
then $J^*(\lambda_1,\lambda_2)$ is defined to be 
the  analytic function space  integral of $F$ 
over $C_{a,b}^2[0,T]$  with parameter $\vec\lambda = (\lambda_1,\lambda_2)$, 
and for $\vec\lambda \in \mathbb C_+^2$ we write 
\[
E^{\mathrm{an}_{\vec\lambda}}[F] 
\equiv E_{\vec x}^{\mathrm{an}_{\vec\lambda}}[F(x_1, x_2)]
\equiv E_{x_1,x_2}^{\mathrm{an}_{(\lambda_1,\lambda_2)}}[F(x_1, x_2)]
= J^*(\lambda_1,\lambda_2).  
\]
Let $q_1$ and $q_2$ be nonzero real numbers. 
Let $F$ be a functional such that
$E^{\mathrm{an}_{\vec\lambda}}[F]$ exists 
for all $\vec\lambda \in \mathbb C_+^2$.
If the following limit exists, 
we call it the generalized analytic Feynman integral of $F$ 
with parameter $\vec q=(q_1,q_2)$ 
and we write
\[
E^{\mathrm{ anf}_{\vec q}}[F]
 \equiv E_{\vec x}^{\mathrm{ anf}_{\vec q}}[F(x_1, x_2)]
 \equiv E_{x_1,x_2}^{\mathrm{ anf}_{(q_1,q_2)}}[F(x_1, x_2)]
  =\lim_{\substack{\vec\lambda \to -i \vec q\\\vec\lambda \in \mathbb C_+^2} } E^{\mathrm{ an}_{\vec\lambda}}[F].
\]
\end{definition}

\setcounter{equation}{0}
\section{A Cameron--Storvick type theorem on $C_{a,b}^2[0,T]$}\label{sec:CS-thm}

In \cite{Cameron51}, 
Cameron (also see \cite[Theorem A]{CS91}) expressed  the Wiener integral 
of the first variation of a functional $F$ on the Wiener space $C_0[0,T]$
in terms of the Wiener integral of the product of $F$ by a linear functional, 
and in \cite[Theorem 1]{CS91}, Cameron and Storvick obtained  a similar 
result for the analytic Feynman integral on $C_0[0,T]$. In \cite[Theorem 2.4]{CSY01},   
Chang, Song and Yoo also obtained a Cameron--Storvick  theorem  on abstract Wiener 
spaces. In \cite{CS03},  Chang and Skoug developed  these results for  functionals 
on the function space $C_{a,b}[0,T]$.

In this section,  we establish a Cameron--Storvick  theorem 
for the generalized analytic Feynman integral of functionals on  the product 
function space $C_{a,b}^2[0,T]$. To do this we first  give the definition of 
the first variation of a functional $F$ on $C_{a,b}^2[0,T]$.


\begin{definition}
Let $F$ be a  functional on $C_{a,b}^2[0,T]$
and let $g_1$ and $g_2$ be functions in   $C_{a,b}[0,T]$.
Then
\begin{equation}\label{eq:1st}
\begin{aligned}
\delta F(x_1,x_2|g_1,g_2)
& =\frac{\partial}{\partial h}\Big(F(x_1+h g_1,x_2) 
+ F(x_1, x_2+h g_2)\Big)\bigg|_{h=0}\\
& =\frac{\partial}{\partial h}F(x_1+h g_1,x_2)\bigg|_{h=0}
+\frac{\partial}{\partial h}F(x_1, x_2+h g_2)\bigg|_{h=0} 
\end{aligned}
\end{equation}
(if it exists) is called the first variation of $F$ 
in the direction of $(g_1, g_2)$.
\end{definition} 

\par
Throughout this section, 
when working with $\delta F(x_1, x_2|g_1,g_2)$,
we will always require $g_1$ and $g_2$ 
to be functions in   $C_{a,b}'[0,T]$.

We first quote  the translation theorem \cite{CS03}  for  the function 
space integral using our notations.


\begin{lemma}[Translation Theorem]\label{thm:tt}
Let $F\in L^1 (C_{a,b}[0,T])$ and let $w_0 \in C_{a,b}'[0,T]$. Then
\begin{equation}\label{eq:translation}
\begin{aligned}
&\int_{C_{a,b}[0,T]}F(x+w_0)d\mu(x)\\
&=\exp\bigg\{-\frac{1}{2}\|w_0\|_{C_{a,b}'}^{2}-(w_0,a)_{C_{a,b}'}\bigg\}
 \int_{C_{a,b}[0,T]}F(x)\exp\{(w_0,x)^{\sim}\}d\mu(x).
\end{aligned}
\end{equation}
\end{lemma}

\begin{theorem}[An integration by parts formula]\label{thm:CS}
Let $g_1$ and $g_2$ be  nonzero  functions in $C_{a,b}'[0,T]$.
Let $F(x_1,x_2)$ be $\mu\times \mu$-integrable over $C_{a,b}^2[0,T]$. 
Assume that   $F$ has  a    first variation
$\delta F (x_1,x_2|g_1,g_2)$ for all $(x_1,x_2)\in C_{a,b}^2[0,T]$  
such that    for some  $\gamma>0$,
\[
\sup_{|h|\le \gamma }
\big|\delta F   (x_1 +h  g_1,x_2 +h  g_2|  g_1,g_2)\big|
\]
is $\mu\times \mu$-integrable over $C_{a,b}^2[0,T]$
as a functional  of $(x_1,x_2)\in C_{a,b}^2[0,T]$.  
Then 
\begin{equation}\label{eq:CA-basic} 
\begin{aligned}
&E_{\vec x}[ \delta F(x_1,x_2|g_1,g_2)]\\
&  \quad  =
 E_{\vec x}\big[ F(x_1, x_2)\big\{(g_1, x_1)^{\sim}
   +(g_2, x_2)^{\sim}\big\}\big]\\
& \qquad 
-\big\{(g_1, a)_{C_{a,b}'}+(g_2, a)_{C_{a,b}'}\big\}
E_{\vec x}[ F(x_1,x_2)  ].
\end{aligned}
\end{equation}
\end{theorem}
\begin{proof}
First note that 
\[
\begin{aligned}
&\delta F (x_1+h  g_1,x_2+h g_2 |  g_1,g_2)\\
& =\frac{\partial}{\partial \lambda}
   F (x_1 +h g_1+\lambda g_1,x_2+h g_2  )\bigg|_{\lambda=0}
+\frac{\partial}{\partial \lambda}
   F (x_1+h g_1, x_2 +h g_2+\lambda g_2  )\bigg|_{\lambda=0}\\
& =\frac{\partial}{\partial \lambda}
   F (x_1 +(h  +\lambda) g_1,x_2+h g_2  )\bigg|_{\lambda=0}
+\frac{\partial}{\partial \lambda}
   F (x_1+h g_1, x_2 +(h +\lambda) g_2  )\bigg|_{\lambda=0}\\
& =\frac{\partial}{\partial \mu}
   F (x_1 + \mu g_1,x_2+h g_2  )\bigg|_{\mu=h}
+\frac{\partial}{\partial \mu}
   F (x_1+h g_1, x_2 +\mu g_2  )\bigg|_{\mu=h}\\
& =2 \frac{\partial}{\partial h}
   F (x_1 + h g_1,x_2+h g_2  ).
\end{aligned}
\]
But since 
\[
\sup_{|h|\le \gamma }
 \bigg| \frac{\partial}{\partial h}
   F (x_1 + h g_1,x_2+h g_2  )\bigg|
\]
is $\mu\times\mu$-integrable,
\[
\frac{\partial}{\partial h}   F (x_1 + h g_1,x_2+h g_2  )
\]
is $\mu\times\mu$-integrable  for sufficiently small values of $h$.
Hence by the Fubini theorem  and equation \eqref{eq:translation},
it follows that 
\[
\begin{aligned}
&E_{\vec x}[ \delta   F(x_1, x_2|g_1,g_2)]\\
& = E_{\vec x}\bigg[ \frac{\partial}{\partial h} \Big(  F(x_1+h g_1, x_2)  
+ F(x_1, x_2 +h g_2 ) \Big)\bigg|_{h=0}\bigg]\\
& =  \frac{\partial}{\partial h} E_{\vec x}\big[ F(x_1+h g_1, x_2)  
+   F(x_1, x_2 +h g_2 ) \big] \bigg|_{h=0} \\
& = \frac{\partial}{\partial h} E_{\vec x}\big[  F(x_1+h g_1, x_2) \big]\bigg|_{h=0}
+   \frac{\partial}{\partial h} E_{\vec x}\big[  F(x_1, x_2+h g_2) \big]\bigg|_{h=0}\\
& = \frac{\partial}{\partial h} E_{x_2}\big[
   E_{x_1}\big[ F(x_1+h g_1, x_2)\big] \big]   \bigg|_{h=0}
  + \frac{\partial}{\partial h} E_{x_1} \big[
   E_{x_2}\big[ F(x_1, x_2 +h g_2)\big]  \big]\bigg|_{h=0} \\
\end{aligned}
\] 
\[
\begin{aligned}
& =\frac{\partial}{\partial h}\bigg(
   \exp\bigg\{-\frac{h^2}{2}\|g_1\|_{C_{a,b}'}^2-h(g_1,a)_{C_{a,b}'}\bigg\}\\
& \qquad\qquad\qquad\qquad\times
E_{x_2}\big[E_{x_1}\big[ F(x_1, x_2)\exp\{h(g_1,x_1)^{\sim}\}\big]\big]\bigg) 
\bigg|_{h=0}  \\
& \quad +
    \frac{\partial}{\partial h } \bigg(
       \exp\bigg\{-\frac{h^2}{2}\|g_2\|_{C_{a,b}'}^2-h(g_2,a)_{C_{a,b}'}\bigg\}\\
& \qquad\qquad\qquad\qquad\times
    E_{x_1} \big[   E_{x_2}\big[F(x_1, x_2)\exp\{h(g_2,x_2)^{\sim}\} \big] \big]
\bigg)\bigg|_{h=0}  \\
& = E_{\vec x}\big[ F(x_1, x_2)\big\{(g_1, x_1)^{\sim}+(g_2, x_2)^{\sim}\big\}\big]\\
& \qquad 
-\big\{(g_1, a)_{C_{a,b}'}+(g_2, a)_{C_{a,b}'}\big\}
E_{\vec x}\big[ F(x_1,x_2)  \big] 
\end{aligned}
\] 
as desired.
\end{proof}


\begin{lemma}\label{lemma:CS}
Let $g_1$, $g_2$,  and $F$
be as in Theorem \ref{thm:CS}.
For each $\rho_1>0$ and $\rho_2 >0$, assume that   
$F(\rho_1 x_1,\rho_2 x_2)$ is $\mu\times\mu$-integrable.
Furthermore assume that $F (\rho_1 x_1, \rho_2 x_2)$ 
has a  first variation
$\delta F(\rho_1 x_1, \rho_2 x_2|\rho_1 g_1,\rho_2 g_2)$ 
for all $(x_1,x_2) \in  C_{a,b}^2[0,T]$  
such that    for some  positive function $\gamma(\rho_1,\rho_2)$,
\[
\sup_{|h|\le \gamma(\rho_1,\rho_2)}
\big|\delta F   (\rho_1 x_1 +\rho_1 h  g_1,\rho_2 x_2 +\rho_2 h  g_2|
      \rho_1 g_1,\rho_2 g_2)\big|
\]
is   $\mu\times \mu$-integrable over $C_{a,b}^2[0,T]$
as a functional  of $(x_1,x_2)\in C_{a,b}^2[0,T]$.  
Then 
\begin{equation} \label{eq:CA-b2}
\begin{aligned}
&E_{\vec x} \big[\delta F(\rho_1 x_1,\rho_2 x_2|\rho_1 g_1,\rho_2 g_2) \big]\\
&  \quad  =
  E_{\vec x}\big[ F(\rho_1 x_1,\rho_2  x_2)  
\big\{(g_1,x_1)^{\sim}+(g_2, x_2)^{\sim}\big\} \big]\\
&\qquad 
-\big\{(g_1,a)_{C_{a,b}'}+(g_2, a)_{C_{a,b}'}\big\}
 E_{\vec x}\big[F(\rho_1 x_1,\rho_2  x_2)   \big].
\end{aligned}
\end{equation}
\end{lemma}
\begin{proof}
Given a pair $(\rho_1,\rho_2)$ with $\rho_1>0$ and $\rho_2>0$,
let $R_{(\rho_1,\rho_2)}(x_1,x_2)=F(\rho_1 x_1,\rho_2 x_2)$.
Then we have that
\[
R_{(\rho_1,\rho_2)}(x_1+ h g_1, x_2)= F(\rho_1 x_1 + \rho_1 h g_1, \rho_2 x_2)
\]
and 
\[
R_{(\rho_1,\rho_2)}(x_1, x_2+ h g_2)= F(\rho_1 x_1, \rho_2 x_2 + \rho_2 h g_2)
\]
and that
\[
\frac{\partial }{\partial h}R_{(\rho_1,\rho_2)}(x_1+ h g_1, x_2)\bigg|_{h=0}
= \frac{\partial }{\partial h}F(\rho_1 x_1 + \rho_1 h g_1, \rho_2 x_2)\bigg|_{h=0}
\]
and 
\[
\frac{\partial }{\partial h}R_{(\rho_1,\rho_2)}(x_1, x_2+ h g_2)\bigg|_{h=0}
= \frac{\partial }{\partial h}F(\rho_1 x_1, 
  \rho_2 x_2 + \rho_2 h g_2)\bigg|_{h=0}.
\]
Thus we have
\[
\begin{aligned}
&\delta F(\rho_1 x_1,\rho_2 x_2|\rho_1 g_1, \rho_2 g_2)\\
&= \frac{\partial }{\partial h}F(\rho_1 x_1 + \rho_1 h g_1, \rho_2 x_2)\bigg|_{h=0}
   +\frac{\partial }{\partial h}F(\rho_1 x_1  , \rho_2x_2+ \rho_2h g_2)\bigg|_{h=0}\\
&= \frac{\partial }{\partial h}R_{(\rho_1,\rho_2)}(x_1+ h w_1, x_2)\bigg|_{h=0}
   +\frac{\partial }{\partial h}R_{(\rho_1,\rho_2)}(x_1, x_2+ h g_2)\bigg|_{h=0} \\
&=\delta R_{(\rho_1,\rho_2)}(  x_1, x_2| g_1, g_2).
\end{aligned}
\]
Hence by equation \eqref{eq:CA-basic} with $F$ replaced with $R_{(\rho_1,\rho_2)}$, we have
\[
\begin{aligned}
&E_{\vec x} \big[ \delta F(\rho_1 x_1,\rho_2 x_2|\rho_1g_1,\rho_2g_2) \big]\\
&=E_{\vec x} \big[ \delta R( x_1, x_2| g_1, g_2) \big]\\
&=E_{\vec x} \big[ R(x_1, x_2)\big\{(g_1,x_1)^{\sim}+(g_2, x_2)^{\sim}\big\}\big]\\
& \qquad 
-\big\{(g_1,a)_{C_{a,b}'}+(g_2, a)_{C_{a,b}'}\big\}
 E_{\vec x} \big[ R(x_1,x_2) \big]\\
&=E_{\vec x} \big[ F(\rho_1 x_1,\rho_2  x_2)\big\{(g_1,x_1)^{\sim}+(g_2, x_2)^{\sim}\big\}\big]\\
& \qquad 
-\big\{(g_1,a)_{C_{a,b}'}+(g_2, a)_{C_{a,b}'}\big\}E_{\vec x} \big[F(\rho_1x_1,\rho_2x_2) \big]
\end{aligned}
\]
which establishes equation \eqref{eq:CA-b2}.
\end{proof}

\begin{theorem} \label{thm:CS-B2} 
Let $g_1$, $g_2$,   and $F$
be as in Lemma \ref{lemma:CS}.
Then if any two of the three generalized  analytic Feynman integrals 
in the following equation exist,
then the third one also exists, and equality holds:
\begin{equation}\label{eq:CSthm}
\begin{aligned}
&E_{\vec x}^{\mathrm{anf}_{(q_1,q_2)}}\big[\delta  F( x_1, x_2|g_1, g_2)\big]\\
&=
-iE_{\vec x}^{\mathrm{anf}_{(q_1,q_2)}} \big[F( x_1, x_2)\big\{q_1(g_1,x_1)^{\sim}
+q_2 (g_2,x_2)^{\sim}\big\}\big]\\
&\quad
-\Big\{(-iq_1)^{1/2}(g_1,a)_{C_{a,b}'}
  +(-iq_2)^{1/2}(g_2,a)_{C_{a,b}'}\Big\}
E_{\vec x}^{\mathrm{anf}_{(q_1,q_2)}}\big[F( x_1, x_2)\big].
\end{aligned}
\end{equation} 
\end{theorem}
\begin{proof}
Let $\rho_1>0$ and $\rho_2>0$ be given.
Let $y_1= \rho_1^{-1}g_1$ and $y_2= \rho_2^{-1}g_2$.
By equation \eqref{eq:CA-b2},
\begin{equation}\label{eq:CS-l}
\begin{aligned}
&E_{\vec x} \big[\delta F(\rho_1 x_1,\rho_2 x_2| g_1,  g_2) \big]\\
&=E_{\vec x} \big[ \delta F(\rho_1 x_1,\rho_2 x_2| 
  \rho_1 y_1, \rho_2 y_2) \big]\\
&= E_{\vec x} \big[ F(\rho_1 x_1,\rho_2  x_2) 
  \big\{(y_1,x_1)^{\sim}+(y_2, x_2)^{\sim}\big\}\big]\\
&\quad 
-\big\{(y_1,a)_{C_{a,b}'}+(y_2, a)_{C_{a,b}'}\big\}
   E_{\vec x} \big[ F(\rho_1 x_1,\rho_2  x_2)  \big]\\
&= E_{\vec x} \big[ F(\rho_1 x_1,\rho_2  x_2) 
\big\{\rho_1^{-2}(g_1,\rho_1x_1)^{\sim}
  +\rho_2^{-2}(g_2, \rho_2x_2)^{\sim}\big\}\big]\\
&\quad 
-\big\{\rho_1^{-1}(g_1,a)_{C_{a,b}'}+\rho_2^{-1}(g_2, a)_{C_{a,b}'}\big\}
E_{\vec x} \big[ F(\rho_1 x_1,\rho_2  x_2) \big].
\end{aligned}
\end{equation}
Now let $\rho_1=\lambda_1^{-1/2}$ and $\rho_2=\lambda_2^{-1/2}$. 
Then equation \eqref{eq:CS-l} becomes
\begin{equation}\label{eq:CS-ll}
\begin{aligned}
&E_{\vec x} \big[ \delta F(\lambda_1^{-1/2} x_1,
   \lambda_2^{-1/2}x_2| g_1,  g_2) \big]\\
&= E_{\vec x} \big[ F(\lambda_1^{-1/2} x_1,\lambda_2^{-1/2}  x_2)   \
\big\{\lambda_1 (g_1,\lambda_1^{-1/2}x_1)^{\sim}
    +\lambda_2 (g_2, \lambda_2^{-1/2}x_2)^{\sim}\big\}  \big]\\
&\quad 
-\big\{\lambda_1^{1/2}(g_1,a)_{C_{a,b}'}
    +\lambda_2^{1/2}(g_2, a)_{C_{a,b}'}\big\}
E_{\vec x} \big[ F(\lambda_1^{-1/2} x_1,\lambda_1^{-1/2}  x_2)  \big].\\
\end{aligned}
\end{equation}
Since $\rho_1>0$ and $\rho_2>0$ were arbitrary,
we have that equation \eqref{eq:CS-ll} 
holds for all $\lambda_1>0$ and $\lambda_2>0$.
We now use  Definition \ref{def:gfi} to obtain
our desired conclusion.
\end{proof}


\begin{corollary}
Let $H$ be a $\mu$-integrable functional on $C_{a,b}[0,T]$.
Let $g$ be a function in $C_{a,b}'[0,T]$.
Then if any two of the three generalized  analytic Feynman integrals  on $C_{a,b}[0,T]$
in the following equation exist,
then the third one also exists, and equality holds:
\[
\begin{aligned}
&E_{x}^{\mathrm{anf}_{q}}\big[\delta H(x|g)\big]\\
&=-iqE_{x}^{\mathrm{anf}_{q}} \big[H(x)(g,x)^{\sim}\big]
- (-iq )^{1/2}(g ,a)_{C_{a,b}'} E_{x}^{\mathrm{anf}_{q}}\big[H(x)\big],
\end{aligned}
\]
where $E_{x}^{\mathrm{anf}_{q}} [H(x) ]=\int_{C_{a,b}[0,T]}^{\mathrm{anf}_{q}}H(x)d\mu(x)$ 
means the generalized analytic Feynman integral of functionals $H$ on
$C_{a,b}[0,T]$, see  \cite{CC16,CCK15,CCS03,CCS10,CS03,CC12}.
\end{corollary}
\begin{proof}
Simply choose $F(x_1,x_2)= H(x_1)$.
\end{proof}

\setcounter{equation}{0}
\section{Functionals in  the  generalized Fresnel type class
$\mathcal{F}_{A_1,A_2}^{\,\,a,b}$}\label{sec:fab}

\par
In this section we introduce  the generalized Fresnel type class 
$\mathcal{F}_{A_1,A_2}^{\,\,a,b}$   to which we apply our  Cameron--Storvick type 
theorem. 

Let $\mathcal{M}(C_{a,b}'[0,T])$ be the space 
of  complex-valued, countably additive (and hence finite)  
Borel measures on $C_{a,b}'[0,T]$. The space 
$\mathcal{M}(C_{a,b}'[0,T])$ is a Banach algebra 
under the total variation norm and with 
convolution as multiplication, see \cite{Cohn,Rudin}.

\begin{definition}\label{def:class}
Let $A_1$ and $A_2$ be bounded, nonnegative self-adjoint operators 
on $C_{a,b}'[0,T]$.
The generalized Fresnel type class  $\mathcal{F}_{A_1,A_2}^{\,\,a,b}$ 
of functionals on $C_{a,b}^2[0,T]$ 
is defined as the space of all functionals $F$ 
on $C_{a,b}^2[0,T]$ of the form
\begin{equation}\label{eq:element2}
F(x_1,x_2)=\int_{C_{a,b}'[0,T]} 
  \exp\bigg\{\sum\limits_{j=1}^2i(A_j^{1/2}w,x_j)^{\sim}\bigg\} d f (w)      
\end{equation}
for s-a.e. $(x_1,x_2)\in C_{a,b}^2[0,T]$, where  $f$ is in $\mathcal M(C_{a,b}'[0,T])$. 
More precisely, since we identify functionals 
which coincide s-a.e. on $C_{a,b}^2[0,T]$,
$\mathcal{F}_{A_1,A_2}^{\,\,a,b}$ can be regarded as 
the  space of all s-equivalence classes 
of functionals of the form \eqref{eq:element2}.
For more details, see \cite{CCS19,CSC13}.
\end{definition}


\begin{remark}\label{re:A}
{\rm (1)} 
Note that in case $a(t)\equiv 0$ and $b(t)=t$ on $[0,T]$,
the function space $C_{a,b}[0,T]$ reduces   to 
the classical Wiener space $C_0[0,T]$. 
In this case, the generalized Fresnel type   class $\mathcal{F}_{A_1,A_2}^{\,\,a,b}$ 
reduces   to  the Kallianpur and  Bromley  Fresnel 
class $\mathcal F_{A_1,A_2}$,  see \cite{KB84}.

{\rm (2)}  In addition, if we choose 
$A_1=I$ (identity operator)  and $A_2=0$ (zero operator),
then   the  class $\mathcal{F}_{A_1,A_2}^{\,\,a,b}$
reduces to the Fresnel class $\mathcal{F}(C_0[0,T])$.
 It is known, see \cite{Johnson82}, that  $\mathcal{F}(C_0[0,T])$ 
forms a Banach algebra over the complex field.

{\rm (3)}  
The map $f \mapsto F$ defined by \eqref{eq:element2} 
sets up an algebra isomorphism 
between $\mathcal M(C_{a,b}'[0,T])$ and  $\mathcal F_{A_1,A_2}^{\,\,a,b}$ 
if $\mathrm{Ran}(A_1+A_2)$ is dense in $C_{a,b}'[0,T]$  
where $\mathrm{Ran}$  indicates the range of an operator.
In this case, $\mathcal{F}_{A_1,A_2}^{\,\,a,b}$ becomes a Banach algebra 
under the norm $\|F\|=\|f\|$. 
For more details,   see  \cite{KB84}. 
\end{remark}

As discussed in \cite[Remark 7]{CSC13}, for a functional $F$ in $\mathcal F_{A_1,A_2}^{\,\,a,b}$ 
and a vector $\vec q=(q_1,q_2)$ with $q_1\ne 0$ and $q_2\ne0$, the generalized analytic Feynman integral 
$E^{\mathrm{ anf}_{\vec q}}[F]$ might not exist. By a simple modification of the example
illustrated in \cite{CC16}, we can construct an example for the functional whose 
generalized analytic Feynman integral  $E^{\mathrm{ anf}_{\vec q}}[F]$ does not exist.
Thus we need to impose additional restrictions on the functionals $F$ in  $\mathcal F_{A_1,A_2}^{\,\,a,b}$.

Given a positive real number $q_0>0$, and 
bounded, nonnegative self-adjoint operators  $A_1$ and $A_2$ on $C_{a,b}'[0,T]$,
let
\begin{equation}\label{eq:Domi-1}
\begin{aligned}
k(q_0;\vec A;w)
&\equiv  k(q_0;A_1,A_2;w)\\
&=\exp\bigg\{  \sum_{j=1}^2(2q_0)^{-1/2} 
 \|A_j^{1/2}\|_{o}\|w\|_{C_{a,b}'}\|a\|_{C_{a,b}'}   \bigg\}
\end{aligned}
\end{equation}
where $\|A_j^{1/2}\|_o$ means the operator norm of $A_j^{1/2}$ 
for $j\in\{1,2\}$.
For the existence of the  generalized analytic Feynman integral  of $F$,
we  define a subclass $\mathcal{F}_{A_1,A_2}^{\, \,q_0}$
of $\mathcal{F}_{A_1,A_2}^{\,\,a,b}$ by
$F\in \mathcal{F}_{A_1,A_2}^{\,\, q_0}$  if and only if 
\[
\int_{C_{a,b}'[0,T]}k(q_0;\vec A;w)d|f|(w)<+\infty   
\]
where $f$ and $F$ are  related by equation \eqref{eq:element2}
and $k$ is given by equation \eqref{eq:Domi-1}.

The following theorem is due to Choi, Skoug and Chang \cite{CSC13}.


\begin{theorem}\label{thm:t1q}
Let $q_0$ be a positive real number  and 
let $F$ be an element of $\mathcal F_{A_1, A_2}^{\,\,q_0}$.
Then for all real numbers $q_1$ and $q_2$ 
with $|q_j|> q_{0}$, $j\in\{1,2\}$,
the generalized analytic Feynman integral 
$E^{\mathrm{anf}_{\vec q}} [F]$ 
of $F$ exists and is given by the formula
\begin{equation}\label{eq:Fab-feynman}
E^{\mathrm{anf}_{\vec q}} [F]  
= \int_{C_{a,b}'[0,T]}\psi(-i\vec q;\vec A;w)df (w),
\end{equation}
where $\psi(-i\vec q;\vec A;w)$ is given by
\begin{equation}\label{eq:notation-psi}
\psi(-i\vec q;\vec A;w)
=\exp\bigg\{\sum_{j=1}^2\bigg[-\frac{i(A_jw,w)_{C_{a,b}'}}{2q_j} 
+i(-iq_j)^{-1/2}(A_j^{1/2}w,a)_{C_{a,b}'}\bigg]\bigg\}.
\end{equation} 
\end{theorem}

\par
For $j\in\{1,2\}$, let $g_j\in C_{a,b}'[0,T]$
and let $F$ be an element of $\mathcal{F}_{A_1,A_2}^{\,\,a,b}$
whose associated measure $f$, see equation \eqref{eq:element2}, 
satisfies the inequality
\begin{equation}\label{eq:finite102}
\int_{C_{a,b}'[0,T]}\|w\|_{C_{a,b}'} d|f|(w)<+\infty.
\end{equation}
Then using equation \eqref{eq:1st}, 
we obtain that 
\begin{equation}\label{eq:evaluation-delta}
\begin{aligned}
&\delta F(x_1,x_2|g_1,g_2)\\
&=\sum_{k=1}^2\bigg[\frac{\partial}{\partial h}\bigg(
 \int_{C_{a,b}'[0,T]}\exp\bigg\{\sum_{j=1}^2i(A_j^{1/2}w,x_j)^{\sim}\\
&\qquad\qquad\qquad\qquad\qquad\qquad\qquad
+ ih(A_k^{1/2}w,g_k)^{\sim}\bigg\} d f(w)\bigg)\bigg|_{h=0}\bigg]\\ 
&= \int_{C_{a,b}'[0,T]}\Big[\sum_{k=1}^2i(A_k^{1/2}w,g_k)_{C_{a,b}'}\Big]
\exp\bigg\{\sum_{j=1}^2i(A_j^{1/2}w,x_j)^{\sim}\bigg\} df(w)\\
&= \int_{C_{a,b}'[0,T]}\exp\bigg\{\sum_{j=1}^2i(A_j^{1/2}w,x_j)^{\sim}\bigg\} 
d \sigma^{\vec A,\vec g}(w)
\end{aligned}
\end{equation}
where the complex measure $\sigma^{\vec A,\vec g}$ 
is defined by
\[
\sigma^{\vec A,\vec g}(B)
= \int_{B}\Big[\sum_{k=1}^2 i(A_{k}^{1/2}w,g_k)_{C_{a,b}'} \Big]df(w), 
 \quad B \in \mathcal{B}(C_{a,b}'[0,T]).
\]
The second equality of \eqref{eq:evaluation-delta}
follows from \eqref{eq:finite102} 
and Theorem 2.27 in \cite{Folland}.
Also, $\delta F(x_1,x_2|g_1,g_2)$ is an element 
of $\mathcal{F}_{A_1,A_2}^{\,\,a,b}$
as a functional of $(x_1,x_2)$,
since by the Cauchy--Schwartz inequality and \eqref{eq:finite102},
\[
\begin{aligned}
\|\sigma^{\vec A,\vec g}\|
&\le \int_{C_{a,b}'[0,T]}\sum_{j=1}^2|i(A_j^{1/2}w,g_j)_{C_{a,b}'}|d|f|(w)\\
&\le \int_{C_{a,b}'[0,T]}\sum_{j=1}^2
   \|A_j^{1/2}\|_{o}\|w\|_{C_{a,b}'}\|g_j\|_{C_{a,b}'}d|f|(w)\\
&\le\bigg(\sum_{j=1}^2\|A_j^{1/2}\|_{o}\|g_j\|_{C_{a,b}'}\bigg)
   \int_{C_{a,b}'[0,T]} \|w\|_{C_{a,b}'} d|f|(w) <+\infty,
\end{aligned}
\]
where $\|A_j^{1/2}\|_o$ is the operator norm of $A_j^{1/2}$.

For the existence of the  generalized analytic Feynman integral  of the first variation $\delta F$
of a functional $F$ in $\mathcal{F}_{A_1,A_2}^{\,\,a,b}$,
we also define a subclass of $\mathcal{F}_{A_1,A_2}^{\, \,a,b}$ as follows: 
given a positive real number $q_0$, 
we  define a subclass $\mathcal{G}_{A_1,A_2}^{\, \,q_0}$ 
of  $\mathcal{F}_{A_1,A_2}^{\,\,a,b}$ by
$F\in \mathcal{G}_{A_1,A_2}^{\,\, q_0}$  
if and only if 
\[
\int_{C_{a,b}'[0,T]}\|w\|_{C_{a,b}'}
k(q_0;\vec A;w)d|f|(w)<+\infty   
\]
where $f$ is the associated measure of $F$  by equation \eqref{eq:element2}
and $k(q_0; \vec A; w)$ is given by equation \eqref{eq:Domi-1}.

\par
Our next theorem  follows quite readily 
from  the techniques developed in the proof of Theorem 9 of \cite{CSC13}.


\begin{theorem}\label{coro:t1q-feynman}
Let $q_{0}$ be a positive real number and 
let $g_1$ and $g_2$ be functions in $C_{a,b}'[0,T]$.
Let $F$ be an element of $\mathcal{G}_{A_1,A_2}^{\, \,q_0}$.
Then for all real numbers $q_1$ and $q_2$ with $|q_j|> q_{0}$, $j\in\{1,2\}$,
the generalized analytic Feynman integral   
of $\delta F( \cdot ,\cdot |g_1,g_2)$ exists 
and  is given by the formula
\begin{equation}\label{eq:gfi-delta}
E_{\vec x}^{\mathrm{anf}_{\vec q}} [\delta F(x_1,x_2|g_1,g_2)]
= \int_{C_{a,b}'[0,T]}\Big[\sum_{j=1}^2i(A_{j}^{1/2}w,g_j)_{C_{a,b}'}\Big]
df_{\vec q}^{\vec A} (w),
\end{equation}
where $f_{\vec q}^{\vec A}$ is a complex measure on $\mathcal B(C_{a,b}'[0,T])$,
the Borel $\sigma$-algebra of $C_{a,b}'[0,T]$,
given by 
\begin{equation}\label{eq:h-function}
f_{\vec q}^{\vec A}(B)
=\int_{B}\psi(-i\vec q;\vec A;w)d f(w),
   \qquad B\in \mathcal{B}(C_{a,b}'[0,T])
\end{equation}
and where $\psi(-i\vec q;\vec A;w)$ is given by \eqref{eq:notation-psi}.
\end{theorem}

\setcounter{equation}{0}
\section{Applications of the Cameron--Storvick type theorem}\label{sec:last}

\par
Let $A$ be a bounded self-adjoint operator on $C_{a,b}'[0,T]$.
Then we can write 
\begin{equation}\label{eq:decomposition}
A=A_+-A_-
\end{equation}
where $A_+$ and $A_-$ are each bounded, nonnegative and self-adjoint.
Take $A_1=A_+$ and $A_2=A_-$ in the definition of 
$\mathcal{F}_{A_1,A_2}^{\,\,a,b}$ above.  
In this section we consider functionals 
in  the generalized Fresnel type class 
$\mathcal{F}_{A}^{\,\,a,b}\equiv \mathcal{F}_{A_+,A_-}^{\,\,a,b}$
where $A$, $A_+$ and $A_-$ are related by 
equation \eqref{eq:decomposition} above.

\par
Let $C_{a,b}^*[0,T]$ be the set of  functions $k$ in $C_{a,b}'[0,T]$ such that $Dk$
is continuous except for a finite number of finite jump discontinuities and is of
bounded variation on $[0,T]$. For any $w\in C_{a,b}'[0,T]$ and $k\in C_{a,b}^*[0,T]$,
let the operation $\odot$ between $C_{a,b}'[0,T]$ and $C_{a,b}^*[0,T]$ be defined by
\begin{equation}\label{def-odot}
w\odot k =D^{-1}(DwDk)    
\end{equation} 
so that $D(w\odot k)=DwDk$, where $DwDk$ denotes the pointwise multiplication of 
the functions $Dw$ and $Dk$.  In equation  \eqref{def-odot}, the operator $D^{-1}$
is given by \eqref{eq:D-inverse} above. Then
$(C_{a,b}^*[0,T],\odot)$ is a commutative algebra  with the identity $b$.
Also we can observe that
for any $w, w_1, w_2\in C_{a,b}'[0,T]$ and $k\in C_{a,b}^*[0,T]$,
\begin{equation}\label{odot-commu}
 w  \odot k =   k \odot w 
\end{equation}
and
\begin{equation}\label{odot-prop}
(w_1, w_2\odot k )_{C_{a,b}'} =(w_1 \odot k, w_2)_{C_{a,b}'}.
\end{equation}
For a more detailed study of the class $C_{a,b}^*[0,T]$, see \cite{CCK15}.

\par
Next let $\vartheta$ be a function in  $C_{a,b}^*[0,T]$ with $D\vartheta=\theta$.
Define an operator $A:C_{a,b}'[0,T]\to C_{a,b}'[0,T]$ by
\begin{equation}\label{eq:final-op}
\begin{aligned}
A w(t)
&=(\vartheta\odot w)(t)
 =\int_0^t D\vartheta(s)D w(s)d b(s)\\
&=\int_0^t \theta(s)\frac{w'(s)}{b'(s)} d b(s) 
 =\int_0^t \theta(s) d w(s).
\end{aligned}
\end{equation}
It is easily shown that  $A$ is a  self-adjoint operator.
We also see that $A=A_+-A_-$ where
\[
A_+w(t)
=\int_0^t \theta^{+}(s)D w(s)d b(s)=\int_0^t \theta^{+}(s) d w(s)
\]
and
\[
A_-w(t)
=\int_0^t \theta^{-}(s)Dw(s)d b(s)=\int_0^t \theta^{-}(s) d w(s)
\]
and where $\theta^+$ and $\theta^-$ are the positive part 
and the negative part of $\theta$, respectively.
Also, $A_+^{1/2}$ and $A_-^{1/2}$ are given by
\[
A_+^{1/2}w(t)
=\int_0^t \sqrt{\theta^{+}}(s)dw(s)
\,\,\hbox{ and }\,\,
A_-^{1/2}w(t)
=\int_0^t \sqrt{\theta^{-}}(s)dw(s),
\]
respectively. 
For a more   detailed study of this decomposition, 
see \cite[pp.187--189]{JL00}.
For notational convenience, let
$\vartheta_+^{1/2} =D^{-1}\sqrt{\theta^+}$
and let 
$\vartheta_-^{1/2} =D^{-1}\sqrt{\theta^-}$.
Then  it follows that
\begin{equation}\label{A+pn}
A_+^{1/2}w(t)
=(\vartheta_{+}^{1/2}\odot w)(t)
\,\,\hbox{ and }\,\,
A_-^{1/2}w(t)
=(\vartheta_{-}^{1/2}\odot w)(t),
\end{equation}
respectively.

\par
For fixed $g\in C_{a,b}'[0,T]$, 
let $g_1=A_+^{1/2}g$ and $g_2=A_-^{1/2}(-g)$.
Then we see that for any functions  $w$ in $C_{a,b}'[0,T]$,
\begin{equation}\label{eq:q001}
\begin{aligned}
&(A_+^{1/2}w, g_1)_{C_{a,b}'}+(A_-^{1/2}w, g_2)_{C_{a,b}'}\\
&=(A_+^{1/2}w, A_+^{1/2}g)_{C_{a,b}'}
 -(A_-^{1/2}w,A_-^{1/2} g)_{C_{a,b}'}\\
&=(A_+w, g)_{C_{a,b}'}-(A_- w, g)_{C_{a,b}'}\\
&=(Aw,g)_{C_{a,b}'}.
\end{aligned}
\end{equation}
Using equations \eqref{A+pn}, \eqref{odot-commu} and \eqref{odot-prop},  we also  see that
\begin{equation}\label{eq:q002}
(A_+^{1/2} w, a)_{C_{a,b}'}
=(\vartheta_{+}^{1/2}\odot w, a)_{C_{a,b}'} 
=(w \odot \vartheta_{+}^{1/2}, a)_{C_{a,b}'} 
=(w, \vartheta_{+}^{1/2}\odot a)_{C_{a,b}'} 
\end{equation}
and
\begin{equation}\label{eq:q003}
(A_-^{1/2} w, a)_{C_{a,b}'}
=(  w, \vartheta_{-}^{1/2}\odot a)_{C_{a,b}'}.
\end{equation}

\par
Assume that  $F$ is an element of 
$\mathcal{F}_{A}^{\,q_0}\cap\mathcal{G}_{A}^{\,q_0}$
for some $q_0\in(0,1)$ 
where $\mathcal{F}_{A}^{\,q_0}\equiv \mathcal{F}_{A_+,A_-}^{\,\,q_0}$
and $\mathcal{G}_{A}^{\,q_0}\equiv \mathcal{G}_{A_+,A_-}^{\,\,q_0}$
(the classes  $\mathcal{F}_{A_1,A_2}^{\,\, q_0}$ and 
$\mathcal{G}_{A_1,A_2}^{\,\, q_0}$  are defined 
in Section \ref{sec:fab}, respectively).

\par
Applying  \eqref{eq:gfi-delta} with $(q_1,q_2)=(1,-1)$   
and $(g_1, g_2)=(A_+^{1/2}g,-A_-^{1/2}g)$, respectively,  \eqref{eq:h-function}, 
and  \eqref{eq:notation-psi} with $(A_1,A_2)=(A_+,A_-)$,  and using \eqref{eq:q001},  
 we  obtain that
\begin{equation}\label{eq:connection}
\begin{aligned}
&E_{\vec x}^{\text{\rm{anf}}_{(1,-1)}}\big[
\delta F(x_1,x_2|A_+^{1/2}g  , -A_-^{1/2}g  )\big]\\
&=E_{\vec x}^{\text{\rm{anf}}_{(1,-1)}}\big[
  \delta F(x_1,x_2|g_1  , g_2  )\big]\\
& =\int_{C_{a,b}'[0,T]}  \Big[i(A_+ ^{1/2}w,g_1)_{C_{a,b}'}
              +i(A_- ^{1/2}w,g_2)_{C_{a,b}'} \Big]\\
& \qquad\quad \times 
\exp\bigg\{   -\frac{i}{2}((A_+-A_-)w,w)_{C_{a,b}'}\bigg\}\\
& \qquad\quad\times
\exp\bigg\{ i\Big[(-i)^{-1/2}(A_+^{1/2}w,a)_{C_{a,b}'}
   +(i)^{-1/2}(A_-^{1/2}w,a)_{C_{a,b}'}\Big] \bigg\} df(w)\\
& =\int_{C_{a,b}'[0,T]} i(Aw,  g)_{C_{a,b}'}
\exp\bigg\{   -\frac{i}{2}(Aw,w)_{C_{a,b}'}\bigg\}\\
&\quad\times
\exp\bigg\{i\Big[(-i)^{-1/2} (A_+^{1/2}w, a)_{C_{a,b}'} 
    +(i)^{-1/2}(A_-^{1/2} w,  a)_{C_{a,b}'} \Big] \bigg\} df(w).
\end{aligned}
\end{equation}
We also see that for all $\rho_1>0$, $\rho_2>0$ and $h\in \mathbb R$
\[
\begin{aligned}
&\big|\delta F(\rho_1x_1+\rho_1 h g_1,  
     \rho_2x_2+\rho_2h g_2| \rho_1g_1,\rho_2g_2 )\big|\\
& \le 
\int_{C_{a,b}'[0,T]} \Big[\big|(A_+ ^{1/2}w, 
       A_+^{1/2}(\rho_1 g))_{C_{a,b}'} \big|
   +\big|(A_- ^{1/2}w, A_- ^{1/2}(-\rho_1 g))_{C_{a,b}'} \big| \Big] 
d |f| (w) \\
& \le 
\rho_1  \int_{C_{a,b}'[0,T]} \big|(A_+ w,  g )_{C_{a,b}'} \big| 
d |f| (w)
   + \rho_2 \int_{C_{a,b}'[0,T]} \big|(A_- w,  g)_{C_{a,b}'}  \big| 
d |f| (w)\\
& \le 
\big(\rho_1 \|A_+\|_o +\rho_2 \|A_-\|_o\big) \|g\|_{C_{a,b}'}
\int_{C_{a,b}'[0,T]}\|w \|_{C_{a,b}'}d |f| (w).
\end{aligned}
\]
But the last expression above
is bounded and is independent of $(x_1,x_2)\in C_{a,b}^2[0,T]$.
Hence 
$\delta F(\rho_1x_1+\rho_1 h g_1,\rho_2x_2+\rho_2 h g_2
    | \rho_1 g_1,\rho_2 g_2 )$
is $\mu\times\mu$-integrable in 
$(x_1,x_2)\in C_{a,b}^2[0,T]$ for every $\rho_1>0$ and $\rho_2>0$.
Also by Theorems \ref{coro:t1q-feynman}    and  \ref{thm:t1q},
the generalized analytic Feynman integrals
$E_{\vec x}^{\mathrm{anf}_{(1,-1)}} 
   [\delta F(x_1,x_2|A_+^{1/2}g,-A_-^{1/2}g)]$
and
$E_{\vec x}^{\mathrm{anf}_{(1,-1)}}[F(x_1,x_2)]$  exist.
Thus using   \eqref{eq:connection} together with  \eqref{eq:q002}  and \eqref{eq:q003},
\eqref{eq:CSthm}, \eqref{eq:q002} and \eqref{eq:q003} with $w$ replaced with $g$,
 it follows that
\[
\begin{aligned}
&\int_{C_{a,b}'[0,T]} i(Aw,  g)_{C_{a,b}'}
\exp\bigg\{   -\frac{i}{2}(Aw,w)_{C_{a,b}'}\bigg\}\\
&\qquad \times
\exp\bigg\{i\Big[(-i)^{-1/2} (w, \vartheta_{+}^{1/2}\odot a)_{C_{a,b}'} 
   +(i)^{-1/2} (w, \vartheta_{-}^{1/2}\odot a)_{C_{a,b}'} \Big] \bigg\} df(w)\\
&=E_{\vec x}^{\text{\rm{anf}}_{(1,-1)}}\big[
\delta F(x_1,x_2)|A_+^{1/2}g  , -A_-^{1/2}g  )\big]\\
&=-iE_{\vec x}^{\mathrm{anf}_{(1,-1)}} 
   \big[F( x_1, x_2)\big\{(A_+^{1/2}g,x_1)^{\sim}
+  (A_-^{1/2}g,x_2)^{\sim} \big\}\big]\\
&\quad
-\Big\{(-i )^{1/2}(A_+^{1/2}g, a)_{C_{a,b}'} 
   -(i)^{1/2}(A_-^{1/2}g,  a)_{C_{a,b}'} \Big\}
E_{\vec x}^{\mathrm{anf}_{(1,-1)}}\big[ F( x_1, x_2) \big]\\
&=-iE_{\vec x}^{\mathrm{anf}_{(1,-1)}} 
   \big[F( x_1, x_2)\big\{(A_+^{1/2}g,x_1)^{\sim}
+  (A_-^{1/2}g,x_2)^{\sim} \big\}\big]\\
&\quad
-\Big\{(-i )^{1/2}(g, \vartheta_{+}^{1/2}\odot a)_{C_{a,b}'} 
   -(i)^{1/2}(g, \vartheta_{-}^{1/2}\odot a)_{C_{a,b}'} \Big\}
E_{\vec x}^{\mathrm{anf}_{(1,-1)}}\big[ F( x_1, x_2) \big].
\end{aligned}
\]
This result together with \eqref{eq:Fab-feynman} yields the formula
\begin{equation}\label{eq:final-display}
\begin{aligned}
&E_{\vec x}^{\mathrm{anf}_{(1,-1)}} 
   \big[F( x_1, x_2)\big\{(A_+^{1/2}g,x_1)^{\sim}
+  (A_-^{1/2}g,x_2)^{\sim}\big\}\big]\\
&=i\Big\{(-i )^{1/2}(g, \vartheta_{+}^{1/2}\odot a)_{C_{a,b}'} 
   -(i)^{1/2}(g, \vartheta_{-}^{1/2}\odot a)_{C_{a,b}'} \Big\}\\
&\quad\times
 \int_{C_{a,b}'[0,T]}\exp\bigg\{-\frac{i}{2}(Aw,w)_{C_{a,b}'}\bigg\} \\
&\qquad\quad\times
\exp\bigg\{i\Big[(-i)^{-1/2} (w, \vartheta_{+}^{1/2}\odot a)_{C_{a,b}'} 
   +(i)^{-1/2} (w,  \vartheta_{-}^{1/2}\odot a)_{C_{a,b}'} \Big] \bigg\} df(w)\\
&\quad
-\int_{C_{a,b}'[0,T]} (Aw,  g)_{C_{a,b}'}
\exp\bigg\{   -\frac{i}{2}(Aw,w)_{C_{a,b}'}\bigg\}\\
&\qquad\quad\times
\exp\bigg\{i\Big[(-i)^{-1/2} (  w,\vartheta_{+}^{1/2}\odot a)_{C_{a,b}'} 
   +(i)^{-1/2} ( w, \vartheta_{-}^{1/2}\odot  a)_{C_{a,b}'} \Big] \bigg\} df(w).
\end{aligned}
\end{equation}


From this, we obtain more explicit formulas as follows.

\noindent\textbf{Step 1.} Under the special setting of $\vartheta$ in 
equation \eqref{eq:final-op}  above,  we first observe the following table:
\begin{table}[ht]  \label{table-c}
\begin{center}
\begin{tabular}{ | c || c | c | c | c | c | c | c | c | c | c | c | c |}
\hline
$\vartheta$ & $\theta$ & $\theta^+$ & $\theta^-$ & $\sqrt{\theta^+}$ & $\sqrt{\theta^-}$ & $\vartheta_+^{1/2}$ & $\vartheta_-^{1/2}$ &     $A$ & $A_+$ & $A_-$ & $A_+^{1/2}$ & $A_-^{1/2}$ \\ \hline \hline
   \,\, $b$ & \,\,\,$1$      & $1$        & $0$        & $1$               & $0$               & $b$                 & $0$                 & \,\,$I$ & $I$   & $0$   & $I$         & $0$         \\ \hline
       $-b$ & $-1$     & $0$        & $1$        & $0$               & $1$               & $0$                 & $b$                 &    $-I$ & $0$   & $I$   & $0$         & $I$        \\ \hline
\end{tabular}
\end{center}
\medskip
\caption{The results from the setting of $\vartheta$ in equation \eqref{eq:final-display}}
\end{table}

\noindent
\textbf{Step 2.}   Using equation \eqref{eq:final-display} above together 
with Table 1 above we obtain the following two explicit formulas:   
\[
\begin{aligned}
&E_{\vec x}^{\mathrm{anf}_{(1,-1)}} 
   \bigg[ (g,x_1)^{\sim} \int_{C_{a,b}'[0,T]}\exp\big \{i(w,x_1)\big \}d f(w)  \bigg]\\
&=i(-i )^{1/2}(g, a)_{C_{a,b}'}\int_{C_{a,b}'[0,T]}\exp\bigg\{-\frac{i}{2}\|w\|_{C_{a,b}'}^2+i(-i)^{-1/2} (w,  a)_{C_{a,b}'}  \bigg\} df(w)\\
&\quad
-\int_{C_{a,b}'[0,T]} (w,  g)_{C_{a,b}'}
\exp\bigg\{   -\frac{i}{2}\|w\|_{C_{a,b}'}^2+i (-i)^{-1/2} (  w, a)_{C_{a,b}'}  \bigg\} df(w) 
\end{aligned}
\]
and 
\[
\begin{aligned}
&E_{\vec x}^{\mathrm{anf}_{(1,-1)}} 
   \bigg[ (g,x_2)^{\sim} \int_{C_{a,b}'[0,T]}\exp\big \{-i(w,x_2)\big \}d f(w) \bigg]\\
&= -i(i)^{1/2}(g,  a)_{C_{a,b}'} 
 \int_{C_{a,b}'[0,T]}\exp\bigg\{ \frac{i}{2}\|w\|_{C_{a,b}'}^2
   +i(i)^{-1/2} (w,  a)_{C_{a,b}'}   \bigg\} df(w)\\
&\quad
+\int_{C_{a,b}'[0,T]} (w,  g)_{C_{a,b}'}
\exp\bigg\{    \frac{i}{2}\|w\|_{C_{a,b}'}^2
     -i(i)^{-1/2} ( w,  a)_{C_{a,b}'}  \bigg\} df(w).
\end{aligned}
\]

\noindent
\textbf{Other setting.} Letting $\vartheta =\sin \big(\frac{\pi b(t)}{ b(T)}\big)$ in equation \eqref{eq:final-display} above, 
it also follows that 
\begin{align*} 
\theta(t)&=D \vartheta (t) =\frac{\pi }{b(T)} \cos \bigg(\frac{\pi b(t)}{ b(T)}\bigg),\\
\sqrt{\theta^+} (t) &=\frac{\pi }{ b(T)} \cos^+ \bigg(\frac{\pi b(t)}{ b(T)}\bigg)=\frac{\pi }{ b(T)} 
\cos  \bigg(\frac{\pi b(t)}{ b(T)}\bigg)\chi_{[0,b^{-1} (\frac{b(T)}{2} ) ] }(t) , \\
\sqrt{\theta^-} (t) &=\frac{\pi }{b(T)} \cos^- \bigg(\frac{\pi b(t)}{b(T)}\bigg)=-\frac{\pi }{ b(T)} 
\cos  \bigg(\frac{\pi b(t)}{ b(T)}\bigg) \chi_{[b^{-1} (\frac{b(T)}{2} ),T ]}(t) ,\\
A w(t)&=\int_0^t \frac{\pi }{b(T)} \cos \bigg(\frac{\pi b(t)}{b(T)}\bigg) d w(t),\\
A_+^{1/2}w(t) 
&=\int_0^t \frac{\pi }{b(T)} \cos^+ \bigg(\frac{\pi b(t)}{b(T)}\bigg) d w(t)\\
&=\int_0^t \frac{\pi }{b(T)} \cos  \bigg(\frac{\pi b(t)}{b(T)}\bigg) \chi_{[0,b^{-1}(\frac{b(T)}{2})]}(t)d w(t),\\
\intertext{and}\\
A_-^{1/2}w(t)
&=\int_0^t \frac{\pi }{b(T)} \cos^- \bigg(\frac{\pi b(t)}{b(T)}\bigg) d w(t)\\
&=-\int_0^t \frac{\pi }{b(T)} \cos  \bigg(\frac{\pi b(t)}{b(T)}\bigg) \chi_{[b^{-1}(\frac{b(T)}{2}),T]}(t)d w(t).
\end{align*}
 Using these we can replace equation \eqref{eq:final-display} 
with the following 8 formulas 
\begin{align*}
(A_+^{1/2}g,x_1)^{\sim}
&=\frac{\pi }{b(T)}\int_0^{b^{-1}(\frac{b(T)}{2})}\cos\bigg(\frac{\pi b(t)}{b(T)}\bigg)  Dg(t)dx_1(t),\\
(A_-^{1/2}g,x_2)^{\sim}
&=-\frac{\pi }{b(T)}\int_{b^{-1}(\frac{b(T)}{2})}^T  \cos  \bigg(\frac{\pi b(t)}{b(T)}\bigg) Dg(t)dx_2(t), \\
(Aw,w)_{C_{a,b}'}
& = \bigg(\frac{\pi }{b(T)}\bigg)^2 \int_0^T \cos^2 \bigg(\frac{\pi b(t)}{b(T)}\bigg)(Dw)^2(t) db(t), \\
(Aw,g)_{C_{a,b}'}
& = \frac{\pi }{b(T)}\int_0^T  \cos \bigg(\frac{\pi b(t)}{b(T)}\bigg) Dw (t)Dg (t) db(t), \\
(g,  \vartheta_{+}^{1/2}\odot a)_{C_{a,b}'} 
&=\frac{\pi }{b(T)}\int_0^{b^{-1}(\frac{b(T)}{2})}  \cos  \bigg(\frac{\pi b(t)}{b(T)}\bigg) Dg (t)Da (t) db(t),\\
(g,  \vartheta_{-}^{1/2}\odot a)_{C_{a,b}'} 
&=-\frac{\pi }{b(T)}\int_{b^{-1}(\frac{b(T)}{2}) }^T  \cos \bigg(\frac{\pi b(t)}{b(T)}\bigg) Dg (t)Da (t) db(t),\\
(w,  \vartheta_{+}^{1/2}\odot a)_{C_{a,b}'} 
&=\frac{\pi }{b(T)}\int_0^{b^{-1}(\frac{b(T)}{2})} \cos  \bigg(\frac{\pi b(t)}{b(T)}\bigg) Dw (t)Da (t) db(t),\\
\intertext{and} 
(w,  \vartheta_{-}^{1/2}\odot a)_{C_{a,b}'} 
&=-\frac{\pi }{b(T)}\int_{b^{-1}(\frac{b(T)}{2}) }^T  \cos \bigg(\frac{\pi b(t)}{b(T)}\bigg) Dw (t)Da (t) db(t).
\end{align*}
In this example, we chose $\vartheta =\sin \big(\frac{\pi b(t)}{ b(T)}\big)$ and, as presented in equation \eqref{eq:final-op},
this function $\vartheta$  gave the operator $A$.

Consider the sequence  $\{e_m\}$ of functions in $C_{a,b}'[0,T]$,
where for each $m\in\mathbb N$, $e_m$ is given by 
\[
e_m(t)=\frac{\sqrt{2b(T)}}{(m-\frac12)\pi} \sin\bigg(\frac{(m-\frac12)\pi}{b(T)}\bigg), \quad t\in [0,T].
\]
One can show that the sequence  $\{e_m\}$ 
is a complete orthonormal set in $C_{a,b}'[0,T]$ and the functions   $\{e_m\}$ are the eigenvectors of the operator
$B: C_{a,b}'[0,T]\to C_{a,b}'[0,T]$ given by
\[
Bw(t)=\int_0^T \min\{b(s),b(t)\}w(s)db(s), \quad s\in [0,T].
\]
It is known that the operator $B$ is a self-adjoint positive definite trace class operator   and is decomposed by
$B=S^*S$ where $S: C_{a,b}'[0,T]\to C_{a,b}'[0,T]$ is the operator given by
\[
Sw(t)=\int_0^t w (s)db(s), \quad s\in [0,T].
\]
Using this, one can easily check that for each $w\in C_{a,b}'[0,T]$,
\[
\int_0^T w^2(s)db(s)=(w,Bw)_{C_{a,b}'}.
\]
For more details, see \cite{CCS07}.
Thus, under these constructions, our results in this paper   can be applicable to
evaluate Feynman integral using Fourier series approximation (see \cite{Fey-Hib}) for functionals of 
generalized Brownian motion paths.



%
%




\end{document}